\documentclass[reqno,centertags, 12pt]{amsart}
\usepackage{amsmath,amsthm,amscd,amssymb,amsfonts}
\usepackage{latexsym,verbatim}

-\marginparwidth \oddsidemargin -4mm \evensidemargin \oddsidemargin

\newcommand{\bbR}{{\mathbb{R}}}

\newcommand{\bbZ}{{\mathbb{Z}}}

\newcommand{\bbT}{{\mathbb{T}}}

\newcommand{\lb}{\label}

\newcommand{\beq}{\begin{equation}}
\newcommand{\eeq}{\end{equation}}
\newcommand{\ba}{\begin{align}}
\newcommand{\ea}{\end{align}}
\newcommand{\eps}{\varepsilon}

\newcounter{smalllist}

\DeclareMathOperator*{\Lip}{Lip} 

\allowdisplaybreaks \numberwithin{equation}{section}

\newtheorem{theorem}{Theorem}[section]

\newtheorem{lemma}[theorem]{Lemma}
\newtheorem{corollary}[theorem]{Corollary}

\theoremstyle{definition}
\newtheorem{definition}[theorem]{Definition}
\newtheorem{example}[theorem]{Example}

\theoremstyle{remark}

\begin{document}

\title[Relaxation Enhancement]
{Relaxation Enhancement by Time-Periodic Flows}

\author{Alexander Kiselev}
\thanks{Department of
Mathematics, University of Wisconsin, Madison, WI 53706; e-mail:
kiselev@math.wisc.edu}
\author{Roman Shterenberg}
\thanks{\text{Department of
Mathematics, University of Wisconsin, Madison, WI 53706; e-mail:
shterenb@math.wisc.edu}}
\author{Andrej Zlato\v s}
\thanks{Department of
Mathematics, University of Chicago, Chicago, IL 60637; email:
zlatos@math.uchicago.edu}

\begin{abstract}
We study enhancement of diffusive mixing by fast incompressible
time-periodic flows. The class of \it relaxation-enhancing \rm flows
that are especially efficient in speeding up mixing has been
introduced in \cite{CKRZ}. The relaxation-enhancing property of a
flow has been shown to be intimately related to the properties of
the dynamical system it generates. In particular, time-independent
flows $u$ such that the operator $u \cdot \nabla$ has sufficiently
smooth eigenfunctions are not relaxation-enhancing. Here we extend
results of \cite{CKRZ} to time-periodic flows $u(x,t)$ and in
particular show that there exist flows such that for each fixed time
the flow is Hamiltonian, but the
resulting time-dependent flow is relaxation-enhancing. Thus we
confirm the physical intuition that time dependence of a flow may
aid mixing. We also provide an extension of our results
to the case of a nonlinear
diffusion model. The proofs are based on a general criterion for the
decay of a semigroup generated by an operator of the form
$\Gamma+iAL(t)$ with a negative unbounded self-adjoint operator
$\Gamma$, a time-periodic self-adjoint operator-valued function
$L(t)$, and a parameter $A\gg 1$.
\end{abstract}

\maketitle

\section{Introduction}

In the present paper we study enhancement of diffusive mixing by
fast incompressible time-periodic flows. We let $u$ be a
time-periodic incompressible (i.e., $\nabla\cdot u=0$) Lipschitz
vector field (flow) on a smooth compact Riemannian manifold $M$,
or on a bounded domain $M\subset\bbR^n$ with $\partial M \in C^2$.
In the latter case we also require $u(x,t)\cdot \hat{n}=0$ for
$(x,t)\in\partial M\times\bbR$. We consider the PDE
\begin{equation}\label{1.0}
\frac{d}{dt}\phi^A(x,t)+Au(x,At)\cdot
\nabla\phi^A(x,t)=\Delta\phi^A(x,t),\qquad \phi^A(x,0)=\phi_0(x)
\end{equation}
on $M$, with Neumann boundary conditions on $\partial M$ if $M$ is a
bounded domain in $\bbR^n$. Here $\Delta$ is the Laplace-Beltrami
operator on $M$ and $\nabla$ is the covariant derivative. We are
interested in the case of fast flows with $A\gg 1$. Note that the
choice of the term $Au(x,At)$ is natural here because all these
flows have the same streamlines --- solutions of $\tfrac d{d t}
X(x,t)=Au(X(x,t),At)$, $X(x,0)=x$, have the same trajectories
for different $A$ but traverse them at different speeds (proportional to $A$).

It is well known that as time tends to infinity, the solution
$\phi^A(x,t)$ tends to its average
\[
\bar{\phi} \equiv \frac{1}{|M|}\int\limits_M \phi^A(x,t)\,d\mu=
\frac{1}{|M|}\int\limits_M \phi_0(x)\,d\mu=\bar\phi_0,
\]
with $|M|$ the volume of $M$ and $\mu$ the volume measure. We would
like to understand how the speed of relaxation to the average
depends on the properties of the flow and determine which flows are
efficient in enhancing this process.

The question of the influence of advection on diffusion is very
natural and physically relevant, and the subject has a long history.
We refer to the recent paper \cite{CKRZ} for a more detailed
overview of the relevant literature. In \cite{CKRZ}, a class of \it
relaxation-enhancing \rm time-independent flows has been introduced,
and a sharp characterization of such flows has been obtained. Our main
goal here is to generalize the results of \cite{CKRZ} to allow
periodic time dependence, and also to provide some interesting
examples. Let us recall the definition of these flows from
\cite{CKRZ}, adjusted to our setting.

\begin{definition} \label{relaxdef}
We say that the incompressible time-periodic flow $u\in
\Lip(M\times\bbR)$ is {\it relaxation-enhancing} if for any
$\tau,\delta>0$ there is $A_0>0$ such that for any $A>A_0$ and any
initial datum $\phi_0\in L^2(M)$ with $\|\phi_0\|_{L^2(M)}=1$, the
solution $\phi^A(x,t)$ satisfies
\begin{equation} \label{1.0a}
\|\phi^A(\cdot,\tau)-\bar\phi_0\|_{L^2(M)}<\delta.
\end{equation}
\end{definition}

{\it Remark.} We note that just as in \cite{CKRZ},
$\|\phi_0\|_{L^2(M)}=1$ can be replaced by $\|\phi_0\|_{L^p(M)}=1$
and the $L^2(M)$-norm in \eqref{1.0a} by the $L^q(M)$-norm (with any
$p,q\in[1,\infty]$) without a change to the class of
relaxation-enhancing flows.
\smallskip

The flow $u$ defines a unitary evolution $\{U(t)\}_{t\in\bbR}$ on
$L^2(M)$ such that for any $\psi\in L^2(M)$,
\begin{equation} \lb{1.0b}
(U(t)\psi)(X(x,t)) \equiv \psi(x)
\end{equation}
with $X(x,t)$ the unique solution to the ODE
\begin{equation} \lb{1.0c}
\frac d{d t} X(x,t)=u(X(x,t),t), \qquad X(x,0)=x.
\end{equation}
That is,
\begin{equation} \lb{1.0d}
\frac d{dt}(U(t)\psi)+u\cdot\nabla(U(t)\psi)=0.
\end{equation}
We also let $U(t,s)\equiv U(t)U(s)^*$ so that $(U(t,s)\psi)(X(x,t))
\equiv \psi(X(x,s))$. Unitarity of the group
$\{U(t,s)\}_{s,t\in\bbR}$ is implied by incompressibility of $u$
which guarantees that $X(\cdot,t)$ is area-preserving. We note that
if $u(x,t)=u(x)$ is time independent, then
$U(t,s)=e^{(-u\cdot\nabla)(t-s)}$.

The main result of this paper is

\begin{theorem}\label{fluid}
Let $M$ be a smooth compact Riemannian manifold. A time p-periodic
incompressible flow $u \in \Lip(M\times\bbR)$ is relaxation
enhancing if and only if the period operator $U(p)$ has no
eigenfunctions in $H^1(M)$ other than the constant function.
\end{theorem}

{\it Remark.} 1. When $u$ is time-independent, then this is the
main result of \cite{CKRZ} (and $U(p)$ can be replaced by
$u\cdot\nabla$ in the statement of the theorem).
\smallskip

2. In the case of time-independent $u$ and $M$ a bounded domain with
Dirichlet boundary conditions, a necessary and sufficient condition
for $u$ to be relaxation-enhancing has been derived earlier in
\cite{BHN} by methods different from \cite{CKRZ} and this paper. In
particular, \cite{BHN} provides estimates on the principal
eigenvalue of the operator $-\Delta +A u\cdot \nabla$ and ties the
behavior of this eigenvalue with short-time evolution corresponding
to \eqref{1.0}. Such a link is currently not available in the case
of compact manifolds or Neumann boundary conditions.
\smallskip

We will now discuss an example showing how important time dependence
of the flow can be for relaxation enhancement. It is an example of a
relaxation-enhancing time-periodic flow that, frozen at each
instance of time, has closed streamlines and is not
relaxation-enhancing as a stationary flow. This shows that
relaxation enhancement can be achieved by flows of relatively simple
structure if time dependence is allowed. This contrasts with the
time independent case, where relaxation-enhancing flows must be
quite complex (which is necessary to ensure purely continuous
spectrum or only rough eigenfunctions of $u \cdot \nabla$).

We call a time-independent flow $u$ on $M=\bbT^{2n}$ {\it
Hamiltonian} if there is a $C^1$-function $H:M\to \alpha\bbT$ (for
some $\alpha>0$ and $\alpha\bbT\equiv[0,\alpha]$ with ends
identified) or $H:M\to \bbR$ such that
$u(x)=(-H_{x_{n+1}}(x),\dots,-H_{x_{2n}}(x),H_{x_1}(x),\dots,H_{x_n}(x))$.
For instance, the flow $u(x)\equiv(0,2)$ on $\bbT^2$ corresponds to
the $2\bbT$-valued Hamiltonian $H(x)=2x_1$. It is easy to see from
Theorem~\ref{fluid} that no stationary Hamiltonian flow can be
relaxation-enhancing. Indeed, any $\psi(x)\equiv\omega(H(x))$ with
$\omega$ a smooth $\alpha$-periodic function is an $H^1(M)$
eigenfunction of $u\cdot\nabla$.  A part of our motivation was the
question of existence of {\it time-periodic} Hamiltonian
relaxation-enhancing flows which we now answer in the affirmative by
providing the following example on the two-dimensional torus. We
note that a stationary incompressible flow on $\bbT^2$ is
Hamiltonian (and has closed streamlines) if and only if its mean
$(\bar u_1,\bar u_2)\equiv \int_{\bbT^2} u(x) dx$ has rationally
dependent coordinates. That is, $\bar u_1$ and $\bar u_2$ are
integer multiples of the same number $\alpha>0$, in which case the
function $H:\bbT^2\to\alpha\bbT$,
\[
H(x_1,x_2)\equiv \int_0^{x_1} u_2(y,0)\,dy - \int_0^{x_2}
u_1(x_1,y)\,dy,
\]
is a Hamiltonian for $u$. Notice that $H\in C^1(\bbT^2;\alpha\bbT)$
because $\int_0^1 u_1(x_1,y)dy=\bar u_1$ and $\int_0^1
u_2(y,x_2)dy=\bar u_2$ for any $x_1,x_2$ due to incompressibility of
$u$, and that a real-valued Hamiltonian exists for $u$ only if $(\bar
u_1,\bar u_2)=(0,0)$.

\begin{example} \label{hamiltonian}
Let $v\in\Lip(\bbT^2)$ be any stationary smooth incompressible
relaxation-enhancing flow, for instance, a flow with a purely
continuous spectrum (see, e.g., \cite{Fayad1,Fayad2}). If $(\bar
v_1,\bar v_2)\equiv \int_{\bbT^2} v(x) dx$ is its mean, then $\bar
v_1,\bar v_2\neq 0$ because $v$ cannot be Hamiltonian. Let $b\equiv
(\bar v_1,0)$ and consider the time-$\bar v_1^{-1}$-periodic flow
$u(x,t)\equiv v(x+bt)-b$. For any fixed time $t$ the flow $u(x,t)$
has mean $(0,\bar v_2)$ and hence is Hamiltonian.

If now $X'(t)=u(X(t),t)$ and $Y'(t)=v(Y(t))$ with any
$X(0)=Y(0)=x\in\bbT^2$, then $Y(t)=X(t)+bt$. This means that $X(\bar
v_1^{-1})=Y(\bar v_1^{-1})$, and so the period operator $U_u(\bar
v_1^{-1})$ for $u$ equals $U_v(\bar v_1^{-1})\equiv
e^{(-v\cdot\nabla)\bar v_1^{-1}}$. Since $U_v(\bar v_1^{-1})$ has
no eigenfunctions in $H^1(\bbT^2)$ because $v$ is
relaxation-enhancing, Theorem \ref{fluid} shows that the flow $u$
is also relaxation-enhancing.
\end{example}

Thus, we have
\begin{theorem}\label{fluid-cor}
There exists a time-periodic smooth incompressible flow $u$ on
$\bbT^2$ which is relaxation-enhancing but for each $t\in\bbR$, the
flow $u(\cdot,t)$ is Hamiltonian.
\end{theorem}

Just as in \cite{CKRZ}, our main result can be formulated and proved
in an abstract Hilbert space setting. Let $\Gamma$ be a
self-adjoint, positive, unbounded operator with a discrete spectrum on
a separable Hilbert space $H$. Let
$0<\lambda_1\leq\lambda_2\leq\dots$ be the eigenvalues of $\Gamma$,
and $e_j$ the corresponding eigenvectors forming an orthonormal
basis in $H$. The (homogeneous) Sobolev space $H^m(\Gamma)$
associated with $\Gamma$ is formed by all vectors $\psi=\sum_j
c_je_j$ such that
\[
\|\psi\|_{H^m(\Gamma)}\equiv \sum\limits_j\lambda_j^m|c_j|^2<\infty.
\]
We use $\langle\cdot,\cdot\rangle$ for the inner product in $H$
and $\|\cdot\|$ and $\|\cdot\|_1$ for the norms in $H$ and in $H^1(\Gamma)$,
respectively. Note that $H^2(\Gamma)$ is the domain $D(\Gamma)$ of $\Gamma$.

Next, we assume that $L(t)$ is a periodic family of self-adjoint
operators on $H$ (without loss of generality assume that the period
is $1$) which satisfies
\smallskip

{\it Condition 1. There is $C_0<\infty$ such that for any $t\in\bbR$ and
any $\psi\in H^1(\Gamma)$ we have
\begin{equation}\label{con1}
\|L(t)\psi\|\leq C_0\|\psi\|_1.
\end{equation}}
\smallskip

Let us also assume that the family $L(t)$ generates a strongly
continuous unitary group $U(t)$ on $H$. That is, for each $\psi_0\in
H$,  $\psi(t)\equiv U(t)\psi_0$ is a weak solution of
\begin{equation}\label{ueq}
\frac{d}{dt}\psi(t)=iL(t)\psi(t), \qquad \psi(0)=\psi_0.
\end{equation}
We let $V\equiv U(1)$ be the (unitary) period operator and
$U(t,s)\equiv U(t)U(s)^*$. Note that due to periodicity of $L(t)$ we
have
\begin{equation}\label{1.2}
U(t,s)=U(t-\lfloor s\rfloor,s-\lfloor s\rfloor)
\end{equation}
for any $s,t\in\bbR$. We will also assume
\smallskip

{\it Condition 2. There is a function $B\in L^\infty_{\rm
loc}(\bbR)$ such that for any $t,s\in\bbR$ and any $\psi\in
H^1(\Gamma)$ we have
\begin{equation}\label{con2}
\|U(t,s)\psi\|_1\leq B(t-s)\|\psi\|_1
\end{equation}}
\smallskip

Notice that
\eqref{con1} and \eqref{con2} together imply that if $\psi_0\in
H^1(\Gamma)$, then $\psi(t)=U(t)\psi_0$  is a classical
solution of \eqref{ueq} and belongs to $H^1(\Gamma)$.

We are now interested in the behavior of the solutions to the
Bochner differential equation
\begin{equation}\label{eqA}
\frac{d}{dt}\phi^A(t)=iAL(At)\phi^A(t)-\Gamma\phi^A(t),\qquad
\phi^A(0)=\phi_0
\end{equation}
with $A\in\bbR$. When $H\equiv L^2(M)\ominus 1$ is the space of
mean-zero functions from $L^2(M)$, $\Gamma\equiv -\Delta$ and
$L(t)\equiv iu(t)\cdot\nabla$ on $H$, then this is exactly
\eqref{1.0}.

%
%

\begin{definition} \label{abstrdef}
We say that the family $L(t)$ is {\it relaxation-enhancing} (with
respect to $\Gamma$) if for any $\tau,\delta>0$ there is $A_0>0$
such that for any $A>A_0$ and any $\phi_0\in H$ with $\|\phi_0\|=1$,
the solution $\phi^A(t)$ satisfies
\begin{equation} \label{1.1}
\|\phi^A(\tau)\|<\delta.
\end{equation}
\end{definition}

We now have the following abstract version of Theorem \ref{fluid}.

\begin{theorem}\label{main}
Assume Conditions 1 and 2. Then the periodic family $L(t)$ is
relaxation-enhancing if and only if the unitary operator $V$ has no
eigenfunctions in $H^1(\Gamma)$.
\end{theorem}

Notice that Theorem \ref{fluid} now follows directly from this
result.

\begin{proof}[Proof of Theorem \ref{fluid}]
As mentioned above, we let $H\equiv L^2(M)\ominus 1$, and
$\Gamma\equiv -\Delta$ and $L(t)\equiv iu(t)\cdot\nabla$ restricted
to $H$. Conditions 1 and 2 are now implied by Lipschitzness of u
with $C_0\equiv \|u\|_\infty$ and $B(t)\equiv e^{|t|\|\nabla
u\|_\infty}$ (see \cite{CKRZ}), and so Theorem \ref{main} gives
Theorem \ref{fluid} for all $\phi_0$ with $\bar\phi_0=0$. Since
$\bar\phi_0$ is conserved by \eqref{1.0}, the result follows.
\end{proof}

The final extension we discuss in this paper is to the case of {\it
porous medium equations}, where $\Delta \phi^A$ is replaced by
$\Delta (\phi^A)^q,$ $q>1.$ We discuss the setting and the result in
Section~\ref{poroussec}.

\smallskip
\noindent {\bf Acknowledgement.} AK and RS have been supported in
part by the NSF-DMS grant 0314129. AZ has been partially supported
by the NSF-DMS grant 0632442. The authors thank Leonid Polterovich
and Lenya Ryzhik for useful discussions.

\section{Proof of Theorem \ref{main}}

In this section we prove Theorem \ref{main}. As in \cite{CKRZ}, we
reformulate \eqref{eqA} as the small diffusion--long time problem
\begin{equation}\label{eq}
\frac{d}{dt}\phi^\epsilon(t)=iL(t)\phi^\epsilon(t)-\epsilon\Gamma\phi^\epsilon(t),\qquad
\phi^\epsilon(0)=\phi_0
\end{equation}
by setting $\eps\equiv A^{-1}$ and rescaling time by a factor of
$1/\epsilon$. Notice that \eqref{1.1} now becomes
\begin{equation} \label{2.1}
\|\phi^\epsilon(\tau/\epsilon)\|<\delta.
\end{equation}

We first note the following existence and uniqueness result from
\cite{CKRZ}.

\begin{lemma}\label{exun}
Assume that Condition 1 is fulfilled. Then for any $\epsilon> 0$ and
$T>0$, there exists a unique solution $\phi^\epsilon(t)$ of the
equation \eqref{eq} on $[0,T]$ with initial data $\phi_0\in
H^1(\Gamma)$. This solution satisfies
\[
\phi^\epsilon(t)\in L^2([0,T],H^2(\Gamma))\cap
C([0,T],H^1(\Gamma)),\qquad \frac{d}{dt}\phi^\epsilon(t)\in
L^2([0,T],H).
\]
\end{lemma}

{\it Remarks.} 1. The proof of Lemma~\ref{exun} is standard and
proceeds by constructing a weak solution using Galerkin
approximations and then establishing uniqueness and regularity. We
refer, for example, to Evans \cite{Ev} where the construction is
carried out for parabolic PDEs. Given Condition 1, this can be
applied verbatim to the general case.
\smallskip

2. The result is also valid for initial data $\phi_0\in H$, but the
solution has rougher properties on intervals containing $t=0$,
namely
\[
\phi^\epsilon(t)\in L^2([0,T],H^1(\Gamma))\cap
C([0,T],H^{-1}(\Gamma)),\qquad \frac{d}{dt}\phi^\epsilon(t)\in
L^2([0,T],H^{-1}(\Gamma)).
\]
Existence of a rougher solution can also be derived from general
semigroup theory, by checking that $iL-\epsilon\Gamma$ satisfies the
conditions of the Hille-Yosida theorem and thus generates a strongly
continuous contraction semigroup in $H$ (see, e.g. \cite{EN}).
\smallskip

\begin{proof}[Proof of Theorem~\ref{main}]
Let us first assume that
$V\psi=e^{iE}\psi$ for some $\psi\in H^1(\Gamma)$, $\|\psi\|=1$.
We will then show that the family $L(t)$ is not
relaxation-enhancing. For $\eps\ge 0$ let $\phi^\epsilon(t)$ be the
solution of \eqref{eq} with $\phi^\epsilon(0)=\psi$. Then we have
\begin{equation}\label{imp}
\left|\frac{d}{dt}\langle\phi^\epsilon(t),\phi^0(t)\rangle\right| =
\epsilon|\langle\Gamma\phi^\epsilon(t),\phi^0(t)\rangle|\leq
\frac\epsilon 2 (\|\phi^\epsilon(t)\|_1^2+\|\phi^0(t)\|_1^2).
\end{equation}
By multiplying equation \eqref{eq} by $\phi^\epsilon(t)$ and
integrating in time we obtain
\begin{equation}\label{imp1}
2\epsilon\int\limits_0^{\infty}\|\phi^\epsilon(t)\|_1^2dt\leq\|\phi^\epsilon(0)\|^2=1.
\end{equation}
Now $V^n\psi=e^{inE}\psi$ and periodicity of $L(t)$ imply
$\phi^0(n+t)=e^{inE}\phi(t)$ for $n\in\bbZ$ and so due to Condition
2,
\begin{equation}\label{imp2}
\int\limits_0^{\tau/\epsilon}\|\phi^0(t)\|_1^2dt =
\sum\limits_{n=0}^{\lfloor\tau/\epsilon\rfloor-1}\int\limits_0^1\|\phi^0(t)\|_1^2dt
+\int\limits_0^{\{\tau/\epsilon\}}\|\phi^0(t)\|_1^2dt\leq
\frac{\tau}\epsilon B_1^2 \|\psi\|_1^2
\end{equation}
with
\begin{equation}\label{2.4}
B_1\equiv \sup_{t\in[0,1]}B(t),
\end{equation}
where $\lfloor x\rfloor$ and $\{x\}$ are the integer and
fractional parts of $x$. Substituting \eqref{imp1} and \eqref{imp2}
into \eqref{imp} we obtain after integration
\[
|\langle\phi^\epsilon(\tau/\epsilon),\phi^0(\tau/\epsilon)\rangle|\geq
\langle\phi^\epsilon(0),\phi^0(0)\rangle- \frac14- \frac\tau 2
B_1^2\|\psi\|_1^2 =\frac34-\frac\tau 2 B_1^2 \|\psi\|_1^2.
\]
Thus for $\tau\leq B_1^{-2}\|\psi\|_1^{-2}$ we have
$\|\phi^\epsilon(\tau/\epsilon)\|\geq1/4$ for any $\epsilon$, and
hence the family $L(t)$ is not relaxation-enhancing.


Let us now assume that none of the eigenfunctions of $V$ belong to
$H^1(\Gamma)$. We will then show that the family $L(t)$ is
relaxation-enhancing. We start with some auxiliary lemmas.

%

\begin{lemma}\label{large}
Suppose that for all $t\in (a,b)$ we have
$\|\phi^\epsilon(t)\|_1^2\geq N\|\phi^\epsilon(t)\|^2$. Then
\[
\|\phi^\epsilon(b)\|^2\leq e^{-2\epsilon
N(b-a)}\|\phi^\epsilon(a)\|^2.
\]
\end{lemma}
\begin{proof}
This follows immediately from
\begin{equation}\label{equality}
\frac{d}{dt}\|\phi^\epsilon\|^2=2\Re\langle\phi^\epsilon,\phi^\epsilon_t\rangle
=-2\langle\phi^\epsilon,\epsilon\Gamma\phi^\epsilon\rangle =
-2\epsilon\|\phi^\epsilon\|_1^2 \le -2\epsilon
N\|\phi^\epsilon(t)\|^2
\end{equation}
and integration in time.
\end{proof}

This lemma shows that as long as the $H^1(\Gamma)$-norm of
$\phi^\epsilon$ stays large, its $H$-norm will decay rapidly
relative to $e^{-\epsilon t}$ (which is what we need to establish
\eqref{2.1}). We next need to consider the case when
$\|\phi^\epsilon(\tau_0)\|_1^2\le N\|\phi^\epsilon(\tau_0)\|^2$ for
some $\tau_0$. First we show that in this case the evolution
\eqref{eq} will stay for some time relatively close (with respect to
$\epsilon$) to the ``free'' evolution $U(t,\tau_0)\phi(\tau_0)$.


\begin{lemma}\label{sdist}
Let $\phi^\epsilon(t)$ and $\phi^0(t)$ be solutions of the
equation \eqref{eq} with
$\phi^\epsilon(\tau_0)=\phi^0(\tau_0)=\phi_0 \in H^1(\Gamma)$. Then for
any $\tau\ge 0$ we have
\[
\|\phi^\epsilon(\tau_0+\tau)-\phi^0(\tau_0+\tau)\|^2\leq\frac\epsilon
2 \|\phi_0\|_1^2\int\limits_0^\tau B(t)^2dt.
\]
\end{lemma}

\begin{proof}
Regularity guaranteed by Conditions 1 and 2 and Lemma~\ref{exun}
allows us to multiply the equation
\[
(\phi^\epsilon-\phi^0)'=iL(t)(\phi^\epsilon-\phi^0)-\epsilon\Gamma\phi^\epsilon
\]
by $\phi^\epsilon-\phi^0$. We obtain
\[
\frac{d}{dt}\|\phi^\epsilon(t)-\phi^0(t)\|^2\leq
2\epsilon(\|\phi^\epsilon(t)\|_1\|\phi^0(t)\|_1-\|\phi^\epsilon(t)\|_1^2)\leq
\frac\epsilon 2\|\phi^0(t)\|_1^2 \le \frac\epsilon 2
B(t-\tau_0)^2\|\phi_0\|_1^2,
\]
with the last inequality using Condition 2. Integration in time now
gives the result.
\end{proof}

We now need to obtain suitable estimates on the free evolution.  We
denote by $P_c$ the orthogonal projection in $H$ on the continuous
spectral subspace of the unitary operator $V$ and by $P_p=I-P_c$ the
orthogonal projection on the pure point spectral subspace of $V$. We
also denote by $P_N$ the orthogonal projection onto the subspace of
$H$ generated by eigenfunctions of $\Gamma$ belonging to eigenvalues
$\lambda_1,\dots,\lambda_N$. Note that $P_N$ is a compact operator
because $\Gamma$ has a discrete spectrum.

\begin{lemma}\label{cont}
Let $C$ be any compact operator. Then the operator norm
\[
\left\|\frac{1}{T}\int_0^T U(t)^*CU(t)P_c dt\right\|\to 0 \quad
\text{as $T\to\infty$}.
\]
\end{lemma}
\begin{proof}
Denote $D=\lfloor T\rfloor$. We have
\[
\left\|\frac{1}{T}\int_0^T U^*(t)CU(t)P_c dt\right\| =
\left\|\int_0^1 \frac{1}{D}\sum\limits_{n=1}^{D}
(V^*)^{n-1}U(t)^*CU(t)V^{n-1}P_c dt\right\| + O(D^{-1}).
\]
By the dominated convergence theorem it is sufficient to prove that for
any $t\in[0,1]$
\[
\left\|\frac{1}{D}\sum\limits_{n=1}^{D}(V^*)^{n-1}U(t)^*CU(t)V^{n-1}P_c\right\|\to
0 \quad \text{as $D\to\infty$}.
\]
The operator $\tilde{C}=U(t)^*CU(t)$ is compact, so we can reduce
the problem to the case of $\tilde{C}$ being rank 1. The proof in
this case is identical to that of Theorem 5.8 in \cite{CFKS} with
integrals replaced by sums.
\end{proof}

Compactness of $P_N$ and $\|P_N U(t)P_c\phi\|^2 = \langle P_c\phi,
U(t)^*P_N U(t)P_c\phi \rangle$ now gives

\begin{corollary}\label{scont}
For any $N$ and $\sigma>0$ there exists $T_c(N,\sigma)$ such that
for any $T\geq T_c(N,\sigma)$ and any $\phi\in H$ with
$\|\phi\|\leq1$ we have
\[
\frac{1}{T}\int_0^T \|P_N U(t)P_c\phi\|^2 dt\leq\sigma.
\]
\end{corollary}

Next we consider the free evolution of $P_p\phi$.

\begin{lemma}\label{point}
Let $K\subset S\equiv  \{\phi\in H:\ \|\phi\|=1\}$ be a compact set.
Consider the set $K_1\equiv  \{\phi\in K:\ \|P_p\phi\|\geq 1/2\}$.
Then for any $\Omega>0$ we can find $N_p(\Omega,K)$ and
$T_p(\Omega,K)$ such that for any $N\geq N_p(\Omega,K)$, any $T\geq
T_p(\Omega,K)$, and any $\phi\in K_1$ we have
\[
\frac{1}{T}\int_0^T\|P_NU(t)P_p\phi\|_1^2dt\geq \Omega.
\]
\end{lemma}

\begin{proof}
Notice that with $D=\lfloor T\rfloor$,
\begin{equation} \label{2.2}
\frac{1}{T}\int_0^T\|P_NU(t)P_p\phi\|_1^2dt\geq \int_0^1
\frac{1}{D+1}\sum\limits_{n=1}^{D} \|P_N
U(t)V^{n-1}P_p\phi\|_1^2\,dt.
\end{equation}
The proof will now follow from

\begin{lemma}\label{auxpoint}
For any fixed $t\in[0,1]$ there are $N_p(t,\Omega,K)$ and
$D_p(t,\Omega,K)$ such that for any $N\geq N_p(t,\Omega,K)$, any
$D\geq D_p(t,\Omega,K)$, and any $\phi\in K_1$ we have
\begin{equation} \label{2.3}
\frac{1}{D+1}\sum\limits_{n=1}^{D}\|P_N U(t)V^{n-1}P_p\phi\|_1^2\geq
2\Omega.
\end{equation}
\end{lemma}

\begin{proof}
Denote by $e^{iE_j}$ the eigenvalues of $V$ (distinct, without
repetitions) and by $Q_j$ the orthogonal projection on the space
spanned by the eigenfunctions corresponding to $e^{iE_j}$. Then
\eqref{2.3} can be rewritten as
\[
\sum_{j,l} \frac{e^{i(E_j-E_l)D}-1}{(e^{i(E_j-E_l)}-1)(D+1)}
\langle \Gamma P_N U(t)Q_j\phi, P_N U(t)Q_l\phi \rangle \geq
2\Omega
\]
with the fraction equal to $D/(D+1)$ when $j=l$. The rest of the
proof is identical to that of Lemma 3.3 from \cite{CKRZ} with $Q_j$
replaced by $U(t)Q_j$ and integrals replaced by sums, provided we
can show that $U(t)Q_j\phi\notin H^1(\Gamma)$ whenever $Q_j\phi\neq
0$. But this is true because if $U(t)Q_j\phi\in H^1(\Gamma)$, then
$VQ_j\phi = U(1,t)U(t)Q_j\phi\in H^1(\Gamma)$ by Condition 2, which
is a contradiction with the assumption that $V$ has no
eigenfunctions in $H^1(\Gamma)$ (unless $Q_j\phi=0$).
\end{proof}

Using \eqref{2.2} and \eqref{2.3}, it is now easy to finish the
proof of Lemma \ref{point}. Indeed, one only needs to choose
$N_p(\Omega,K)$ and $T_p(\Omega,K)$ to be larger than
$N_p(t,\Omega,K)$ and $D_p(t,\Omega,K)$ for all $t\in E$ with
$E\subset[0,1]$ some set of measure $1/2$. This is possible because
$N_p(t,\Omega,K)$ and $D_p(t,\Omega,K)$ are finite for each $t$.
\end{proof}

We can now proceed with the proof of Theorem \ref{main}. Recall that
we assume that $V$ has no eigenfunctions in $H^1(\Gamma)$. Given
$\tau,\delta>0$, we choose $M$ large enough, so that
$e^{-\lambda_M\tau/160}<\delta$. Define the sets $K\equiv \{\phi\in
S:\ \|\phi\|_1^2\leq B_1^2\lambda_M\}\subset S$ and as before,
$K_1\equiv \{\phi\in K:\ \|P_p\phi\|\geq 1/2\}$ (recall that $B_1$
is from \eqref{2.4}). Choose $N$ so that $N\geq M$ and $N\geq
N_p(5\lambda_M,K)$ from Lemma~\ref{point}. Define
$$
\tau_1\equiv
\max\{T_p(5\lambda_M,K),T_c(N,\frac{\lambda_M}{20\lambda_N}),1\},
$$
where $T_p$ is from Lemma~\ref{point} and $T_c$ from Corollary~\ref{scont}. Finally, choose $\epsilon_0>0$
so that $\tau_1+1<\tau/2\epsilon_0$, and
\[
\epsilon_0 \int_0^{\tau_1+1} B(t)^2 dt\leq\frac{1}{20\lambda_N},
\]
where $B(t)$ is from Condition 2.

Take any $\epsilon<\epsilon_0$. If we have $\|\phi^\epsilon
(s)\|_1^2\geq\lambda_M\|\phi^\epsilon (s)\|^2$ for all
$s\in[0,\tau/2\epsilon]$ then Lemma~\ref{large} implies that
$\|\phi^\epsilon (\tau/2\epsilon)\|^2\leq
e^{-\lambda_M\tau}\leq\delta$ by the choice of $M$ and we are done.
Otherwise, let $\tau_0$ be the first time in the interval
$[0,\tau/2\epsilon]$ such that $\|\phi^\epsilon
(\tau_0)\|_1^2\leq\lambda_M\|\phi^\epsilon (\tau_0)\|^2$. We now let
$\phi^0(t)\equiv U(t,\tau_0)\phi^\eps(\tau_0)$ solve \eqref{eq} with
initial condition $\phi^0(\tau_0)=\phi^\epsilon (\tau_0)$.
Lemma~\ref{sdist} then gives
\begin{equation}\label{impdist}
\|\phi^\epsilon(t)-\phi^0(t)\|^2\leq\frac{\lambda_M}{40\lambda_N}\|\phi^\epsilon(\tau_0)\|^2
\end{equation}
for all $t\in[\tau_0,\lceil\tau_0\rceil+\tau_1]$. We also have
\[
\|\phi^0 (\lceil \tau_0\rceil)\|_1^2\leq
B_1^2\lambda_M\|\phi^\epsilon (\tau_0)\|^2 =  B_1^2\lambda_M\|\phi^0
(\lceil \tau_0\rceil)\|^2
\]
by Condition 2 and so $\phi^0 (\lceil \tau_0\rceil)/\|\phi^0 (\lceil
\tau_0\rceil)\|\in K$. We now claim that the following estimate
holds:
\begin{equation}\label{important}
\|\phi^\epsilon (\lceil\tau_0\rceil+\tau_1)\|^2\leq
e^{-\lambda_M\epsilon\tau_1/20}\|\phi^\epsilon (\tau_0)\|^2.
\end{equation}
Indeed, given our choice of $\tau_1$, Corollary \ref{scont}, Lemma
\ref{point}, \eqref{impdist}, and
$U(t+\lceil\tau_0\rceil,\lceil\tau_0\rceil)=U(t)$, the proof is the
same as that of the almost identical estimate (3.8) in \cite{CKRZ}
(which has $\tau_0$ in place of $\lceil\tau_0\rceil$). Then we have
\begin{equation}\label{five}
\|\phi^\epsilon(\tau_0+\tau_1+1)\|^2\leq\|\phi^\epsilon(\lceil\tau_0\rceil+\tau_1)\|^2
\leq e^{-\lambda_M\epsilon\tau_1/20}\|\phi^\epsilon(\tau_0)\|^2 \leq
e^{-\lambda_M\epsilon(\tau_1+1)/40}\|\phi^\epsilon(\tau_0)\|^2,
\end{equation}
where we used \eqref{equality} in the first inequality. The same
method can be repeated with $\tau_0$ replaced by the first time
after $\tau_0+\tau_1+1$ at which $\|\phi^\epsilon
(t)\|_1^2\leq\lambda_M\|\phi^\epsilon (t)\|^2$, etc. On the other
hand, for any interval $I=[a,b]$ such that
$\|\phi^\epsilon(t)\|^2_1\geq\lambda_M\|\phi^\epsilon(t)\|^2$ on
$I$, we have by Lemma~\ref{large} that
\begin{equation}\label{six}
\|\phi^\epsilon(b)\|^2\leq e^{-2\lambda_M\epsilon
(b-a)}\|\phi^\epsilon(a)\|^2.
\end{equation}
Combining all the decay factors gained from \eqref{five} and
\eqref{six}, and using $\tau_1+1<\tau/2\epsilon$, we find that there
exists $\tau_2\in[\tau/2\epsilon,\tau/\epsilon]$ such that
\[
\|\phi^\epsilon(\tau_2)\|^2\leq e^{-\lambda_M\epsilon\tau_2/40}\leq
e^{-\lambda_M\tau/80}<\delta^2
\]
by our choice of $M$. Then \eqref{equality} gives
$\|\phi^\epsilon(\tau/\epsilon)\|\leq\|\phi^\epsilon(\tau_2)\|<\delta$,
thus finishing the proof of Theorem~\ref{main}.
\end{proof}

\section{Relaxation for the porous medium equation}\label{poroussec}

In this section, we indicate how to generalize our results on
relaxation enhancement to some nonlinear equations. The arguments of
the previous section and \cite{CKRZ} are sufficiently robust to
remain applicable in this more general setting. Here we focus on the
case of the {\it porous medium equation with advection}
\begin{equation}\label{pm}
\frac{d}{dt}\phi^A(x,t)+Au(x,At)\cdot
\nabla\phi^A(x,t)=\Delta(\phi^A(x,t))^q,\qquad \phi^A(x,0)=\phi_0(x),
\end{equation}
with $q>1$ and on a smooth compact Riemannian manifold $M$ without
boundary. We restrict our considerations to initial data $\phi_0$
which are positive and bounded: $0<h \leq \phi_0(x) \leq h^{-1}.$
This is the physically relevant case, and such choice of data also
ensures uniform parabolicity. We refer to \cite{Vaz} (which mainly
concentrates on \eqref{pm} without the advection term) for the
overview of history, basic properties, and applications of the
equation \eqref{pm}. In particular, a unique classical solution to
\eqref{pm} exists under our assumptions on the initial data provided
$u\in C^\infty(M\times\bbR)$ (see \cite[Section 3.1]{Vaz} and
references therein).

We again define relaxation-enhancing flows via Definition
\ref{relaxdef} but this time with the initial data also satisfying $h\le \phi_0\le
h^{-1}$ for some $h>0$, and $A_0$ can additionally depend on $h$.
Notice that the mean $\bar\phi=\bar\phi_0$ of the solution is again
preserved by the evolution \eqref{pm}. We now have

\begin{theorem}\label{porous}
Let $M$ be a smooth compact Riemannian manifold. Consider equation
\eqref{pm} with real-valued positive initial data bounded away from
0 and $\infty$. A time p-periodic incompressible flow $u \in
C^\infty(M\times\bbR)$ is relaxation enhancing for \eqref{pm} if and
only if the period operator $U(p)$ has no eigenfunctions in $H^1(M)$
other than the constant function.
\end{theorem}

{\it Remarks.} 1. The same result also holds in the case of the {\it
generalized porous medium equation with advection}
\begin{equation}\label{gpm}
\frac{d}{dt}\phi^A(x,t)+Au(x,At)\cdot
\nabla\phi^A(x,t)=\Delta\Psi(\phi^A(x,t)),\qquad
\phi^A(x,0)=\phi_0(x),
\end{equation}
where $\Psi$ is any smooth increasing function with $\Psi(0)=0$ and
$\Psi'$ bounded away from zero on each interval $[h,h^{-1}]$, $h>0$.
\smallskip

2. Similarly to Theorem~\ref{fluid}, this theorem can be stated in a
more abstract form. We do not pursue this more general formulation
here since it requires a number of technical assumptions. The role
of the $H^1(\Gamma)$-norm is then typically played by an expression
derived from the nonlinear term. In our case this expression is
$\int \psi^{q-1}|\nabla \psi|^2\,dx$ which is equivalent to the
$H^1(\Gamma)$-norm for all $h\le\psi\le h^{-1}$ (and $h\le
\phi^\eps(t)\le h^{-1}$ is guaranteed by $h\le \phi_0\le h^{-1}$ and
the maximum principle).
\smallskip

\begin{proof}
Most of the proof is parallel to that of Theorem~\ref{fluid} and
Theorem~\ref{main}, so we just indicate the necessary changes. Let
us switch to the equivalent small-diffusion formulation
\begin{equation}\label{pm1}
\frac{d}{dt}\phi^\epsilon(x,t)+u(x,t)\cdot
\nabla\phi^\epsilon(x,t)=\epsilon\Delta(\phi^\epsilon(x,t))^q,\qquad
\phi^\epsilon(x,0)=\phi_0(x).
\end{equation}

If $\psi$ is a nonconstant $H^1$ eigenfunction of $U(p)$, assume that $\psi$ is
bounded by $M<\infty$ (otherwise consider ${\rm
arg}(\psi)\min\{|\psi|,M\}$ instead, which is an $H^1$ eigenfunction
of $U(p)$ with the same eigenvalue). Without loss of generality
assume $\Re \psi\neq 0$ and define $\phi_0\equiv m(\Re\psi+2M)$
where $m>0$ is such that $\|\phi_0-\bar\phi_0\|=1$. Now $h\le
\phi_0\le h^{-1}$ for some $h>0$, and let $\phi^0(t)$ and
$\phi^\epsilon(t)$ solve \eqref{pm1}. It is easy to see that
$\phi^0(t)=m(\Re \psi^0(t)+2M)\ge h$, where $\psi^0(t)$ solves
\eqref{pm1} with $\eps=0$ and initial condition $\psi$. As a result
we have
\[
\|\nabla\phi^0(t)\|\le m \|\nabla\psi^0(t)\|\le mB_1
\|\nabla\psi^0(\lfloor t\rfloor)\| = mB_1 \|\nabla\psi\|.
\]
Instead of \eqref{imp} in the proof of Theorem~\ref{main} we now
obtain
\[ \left| \frac{d}{dt} \langle \phi^\epsilon(t)-\bar\phi_0, \phi^0(t)-\bar\phi_0 \rangle
\right| \leq \frac{\epsilon q}{2}\left( \int_M (\phi^\epsilon)^{q-1}
|\nabla \phi^\epsilon |^2\,dx + \int_M (\phi^\epsilon)^{q-1} |\nabla
\phi^0 |^2\,dx \right).
\]
Similarly to \eqref{imp1}, we have
\[
2\epsilon q \int_0^\infty \int_M (\phi^\epsilon)^{q-1} |\nabla
\phi^\epsilon |^2\,dxdt \leq 1.
\]
Since $|\phi^\epsilon| \leq h^{-1},$ \eqref{imp2} carries over
without changes and we obtain
\[
|\langle \phi^\epsilon(\tau/\epsilon) -\bar\phi_0, \phi^0
(\tau/\epsilon) -\bar\phi_0 \rangle| \geq 1 -\frac14 - \frac{\tau
q}{2}h^{1-q} m^2B_1^2\|\nabla\psi\|^2,
\]
from which lack of relaxation enhancement follows.

The only argument in the proof of the opposite implication that
requires a slight adjustment is  Lemma~\ref{sdist}, where we now
have
\begin{align*}
\frac{d}{dt}\|\phi^\epsilon(t) - & \phi^0(t)\|^2  \leq 2\epsilon
\int_M \Delta (\phi^\epsilon)^q (\phi^\epsilon - \phi^0)\,dx  \\
& \le 2\epsilon q \left( \int_M (\phi^\epsilon)^{q-1}
|\nabla\phi^\epsilon|^2 dx \right)^{1/2} \left( \int_M
(\phi^\epsilon)^{q-1} |\nabla\phi^0|^2 dx \right)^{1/2} - 2\epsilon
q \int_M (\phi^\epsilon)^{q-1} |\nabla\phi^\epsilon|^2 dx \\
& \le \frac{\epsilon q}{2} \int_M (\phi^\epsilon)^{q-1}
|\nabla\phi^0|^2 dx \\
& \leq \frac{\epsilon q h^{1-q}}{2} B(t-\tau_0)^2\|\nabla\phi_0\|^2.
\end{align*}
The rest of the proof involves only estimates on the linear dynamics
with $\eps=0$. Thus all the bounds on the $H^1(\Gamma)$-norm from the proof
of Theorem~\ref{main} translate immediately into estimates on the
decay rate for $\|\phi^\eps-\bar\phi_0\|$, with possibly $h$-dependent
constants.
\end{proof}

\end{document}